\documentclass{amsart}[12pt]
\usepackage[utf8]{inputenc}
\usepackage{tikz}
\usepackage{tikz-cd}
\usetikzlibrary{arrows,automata}

\title{The Minkowski sum of linear Cantor sets}
\author[K. G. Hare]{Kevin G. Hare}
\address{Department of Pure Mathematics, University of Waterloo, Waterloo, Ontario, Canada N2L 3G1}
\email{kghare@uwaterloo.ca}
\thanks{Research of K.G. Hare was supported by NSERC Grant 2019-03930}
\author[N. Sidorov]{Nikita Sidorov}
\address{Department of Mathematics, The University of Manchester, Manchester, M13 9PL, United Kingdom}
\email{sidorov@manchester.ac.uk}

\usepackage{comment}
\usepackage{amssymb,latexsym}
\usepackage{amsmath}
\usepackage{amsthm}
\usepackage{float}
\usepackage{graphicx}
\restylefloat{table}
\restylefloat{figure}

\newtheorem{thm}{Theorem}[section]

\newtheorem{example}[thm]{Example}

\newtheorem{remark}[thm]{Remark}
\newtheorem{corollary}[thm]{Corollary}
\newtheorem{definition}[thm]{Definition}

\begin{document}
\begin{abstract}
Let $C$ be the classical middle third Cantor set.
It is well known that $C+C = [0,2]$ (Steinhaus, 1917). (Here $+$ denotes the Minkowski sum.)
Let $U$ be the set of $z \in [0,2]$ which have a unique representation as
    $z = x + y$ with $x, y \in C$ (the set of uniqueness).
It isn't difficult to show that $\dim_H U =  \log(2) / \log(3)$ and $U$
    essentially looks like $2C$.

Assuming $0,n-1 \in A \subset \{0,1,\dots,n-1\}$,
define $C_A = C_{A,n}$ as the linear Cantor set which the attractor of the iterated function system
\[ \{ x \mapsto (x + a) / n: a \in A \}. \]
We consider various properties of such linear
    Cantor sets. Our main focus will be on the structure of $C_{A,n}+C_{A,n}$
    depending on $n$ and $A$ as well as the properties of the set of uniqueness $U_A$.
\end{abstract}
\maketitle

\section{Introduction}

The history of Minkowski sums of Cantor sets is rich. The most famous result in this area is Hall's theorem stating that $CF_4+CF_4=[0,2]$, where $CF_n$ is the set of real numbers whose continued fraction expansion contains only partial quotients $\{1,\dots,n\}$. It is also known that $CF_3+CF_3\neq[0,2]$. For more details see \cite{RS}.

Questions concerning the addition or multiplication of Cantor sets have long
    been of interest -- see for example \cite{ART, CHM, DUB, NP, Prince, Randolph}.
The present paper is concerned with the case when we add a {\em linear} Cantor set to itself, i.e., a set of $n$-ary expansions with deleted digits.

For $A \subset \mathbb{Z}$ and $n \geq 2$ define
    define $C_{A,n}$ as the linear Cantor set satisfying the IFS
    \[ \{ x \mapsto (x + a) / n: a \in A \}. \]
An alternative but equivalent definition is
    \[ C_{A,n} = \left\{\sum_{i=1}^\infty \frac{a_i}{n^i}: a_i \in A\right\}.\]
Often we restrict our attention to $A$ such that
    $0,n-1 \in A \subset \{0, 1, \dots, n-1\}$.
In such cases we will write $C_A = C_{A,n}$.

We say that $0,n-1 \in A \subset \{0,1,\dots,n-1\}$ is {\em $n$-good} if $C_A + C_A = [0,2]$.
It is worth noting that $C_A + C_A = C_{A+A, n}$ where $A + A = \{a + b: a, b \in A\}$.
It is further worth noting that $C_{A+A, n}$ is an example of an IFS satisfying the
    {\em finite type condition} (see \cite{NgaiWang01}).
As such we see that $\dim_H(C_A+C_A) = 1$ if and only if $C_A + C_A$ contains an interval.

Define $U_A$ as the set of $z \in C_A + C_A$ that has a unique representation as $z = x + y$ for $x, y \in C_A$.

\begin{example}
\label{ex:Cantor}
We have that $A = \{0,2\}$ is $3$-good.
To see this we observe that
    \[ C_A = \left\{\sum \frac{a_i}{3^i}: a_i \in \{0,2\}\right\}. \]
This gives us that
\begin{align*}
C_A + C_A & = \left\{\sum \frac{a_i + b_i}{3^i}: a_i, b_i \in \{0, 2\} \right\} \\
          & = \left\{\sum \frac{c_i}{3^i}: c_i \in \{0, 2, 4\} \right\} \\
          & = \left\{2\cdot \sum \frac{c_i}{3^i}: c_i \in \{0, 1, 2\} \right\} \\
          & = \{2\cdot x: x \in [0,1]\} \\
          & = [0,2]
\end{align*}

This construction also allows us to observe when a representation
   in $C_A + C_A$ is not unique.
Namely if $z = \sum \frac{c_i}{3^i}$ with $c_i \in \{0,2,4\}$ we note that
    the representation will not be unique if any $c_i = 2$
    (as we have $c_i = 2 = a_i + b_i = 0 + 2 = 2 + 0$).
In addition, we see that if the $c_i$ are eventually constant and $0$,
    then it is only unique if $c_i = 0$ for all $i$.
Similarly, if the $c_i$ are eventually constant and $4$.

Hence we have that $U_A$ is a set of dimension $\log(2) / \log(3)$.
In fact, it is a subset of $2 C_A$ where we remove the countable
    set of points of the form $k / 3^s$ with $0 < k < 3^s$.
That is
\[ U_A = \left\{\sum \frac{c_i}{3^i}: c_i \in \{0, 4\}, \{c_i\} \text{ not eventually constant} \right\} \cup \{0,2\}. \]

In the other direction, we see that almost all $z \in [0,2]$ have infintely
    many $2$s in their base $3$ expansion with digits $\{0,2,4\}$.
As such, almost all $z \in [0,2]$ have a continuum of representations
    $z = x + y$ with $x, y \in C_A$.
\end{example}

The observation above that $C_{\{0,2\}} + C_{\{0,2\}} = [0,2]$ is well known,
    first being showed in 1917 by Steinhaus \cite{Steinhaus}.
The solution presented above follows that of Shallit \cite{Shallit}.
The observation that $C_A + C_A = C_{A+ A, n}$ will be used through this
    paper.

It is clear that if $A = \{0,1,\dots, n-1\}$ then $A$ is $n$-good.
It is further clear that if $A = \{0, n-1\}$ and $n \geq 4$ then $A$ is
    not $n$-good.
This raises two obvious questions:
\begin{itemize}
\item How small can $A$ be if $A$ is $n$-good?
\item How large can $A$ be if $A$ is not $n$-good?
\end{itemize}

The first of these questions is the main focus of Section~\ref{sec:small A}.
In Theorem \ref{thm:lower bound good} we show that if $A$ is smaller then
   $\mathcal{O}(\sqrt{n})$ is size, then $A$ is not $n$-good.
Further, in Theorem~\ref{thm:upper bound good} we show that this
    bound is tight, giving an construction
    of an $A$ of size $\mathcal{O}(\sqrt{n})$ which is $n$-good.

\begin{remark}
\label{rem:not n good}
It is not hard to show that
    $A = \{0,3,4,5,\dots, n-1\}$ is not $n$-good.
Hence there exist sets $A$ of size $\mathcal{O}(n)$ which
    are not $n$-good.
\end{remark}

Consider $U_A$, the set of $z \in C_A + C_A$ with unique representation
    as $z = x + y$ with $x, y \in C_A$.
For $A = \{0,2\}$ and $n = 3$ we have that $A$ is $n$-good and
    $U_A$ essentially looks like a middle third Cantor set.
For $A = \{0,1,2\}$ and $n=3$ we instead have that $C_A = [0,1]$ and
    hence $U_A = \{0,2\}$.

Heuristically, the smaller $A$ is,
    the more likely we are to having something non-trivial in $U_A$.
This raises a few additional natural questions.
\begin{itemize}
\item How small can $A$ be with $A$ being $n$-good and $U_A = \{0,2\}$?
\item How large can $A$ be with $A$ being $n$-good and $\dim_H(U_A) > 0$?
\item Does there exist $U_A \neq \{0,2\}$ with
    $\dim_H(U_A) = 0$.
\item How large can we make $\dim_H(U_A)$ if $A$ is $n$-good?
\end{itemize}

We present a construction in Corollary~\ref{cor:lower bound small UA} where
    $A$ is $n$-good, $U_A = \{0,2\}$ is trivial and $A$ has size
    $\mathcal{O}(\sqrt{n})$.
This is best possible asymptotically, and any set $A$ with asymptotically
    smaller size would not be $n$-good.

It was surprising to show in Theorem~\ref{thm:dichotomy} that there is a clear
    dividing line between trivial and non-trivial $U_A$.
That is, either $U_A = \{0,2\}$ or $\dim_H(U_A) \geq \log(2)/\log(n) > 0$.
In particular, there does not exist a countable $U_A$.
In Corollary~\ref{cor:upper bound large UA} we give a construction of
    $A$ where $A$ is $n$-good, $\dim_H(U_A) > 0$ is non-trivial
    and $A$ is size $\mathcal{O}(n)$.
This is best possible asymptotically, as the maximal size $A$ can be is
    $\mathcal{O}(n)$.
These are found in Section~\ref{sec:small UA}.

If $z \in U_A$ has a unique
    representation as $z = x + y$ with $x, y\in C_A$, then
    $x = y$.
Hence $U_A \subset 2 C_A$.
This gives that $\dim_H(U_A) \leq \dim_H(C_A)$.
If $A' \subset A$ where both $A'$ and $A$ are
    $n$-good, we have the inequalities

\begin{center}
\begin{tikzpicture}[node distance=1.2cm]
  \node (A) {$\dim_H(U_{A'})$};
  \node (B) [right of=A] {$\geq$};
  \node (C) [right of=B] {$\dim_H(U_{A})$};
  \node (D) [below of=A] [rotate=90] {$\geq$};
  \node (E) [below of=D] {$\dim_H(C_{A'})$};
  \node (F) [right of=E] {$\leq$};
  \node (G) [right of=F] {$\dim_H(C_{A})$};
  \node (H) [below of=C] [rotate=90] {$\geq$};
\end{tikzpicture}
\end{center}

For the last question of how large can $U_A$ be, we only have partial results.
This is the main topic of Section~\ref{sec:large UA}.
We see that if $\dim_H(C_A) = 1$ then $C_A = [0,1]$ and $U_A = \{0,2\}$.
As $\dim_H(U_A) \leq \dim_H(C_A)$ this gives us that $\dim_H(U_A) \leq \log(n-1)/\log(n) < 1$.
(In fact we can improve this slightly, but not significantly with a bit more analysis.)
Computationally it appears that $\dim_H(U_A) \leq \frac{\log(2)}{\log(3)}$
    with equality only if $n = 3^k$ for some $k$.
In fact, this second observation still appears to be true, irrespective of
     whether $A$ is $n$-good.
See Remark~\ref{rem:UA upper} and Figure~\ref{fig:UA upper}.
We show in Corollary~\ref{cor:UA upper} that for all
    $\varepsilon > 0$ and all $n$ sufficiently large that
    we can construct an $A$ which is $n$-good
    and such that $\dim_H(U_A) \geq \frac{\log(2)}{\log(3)} -\varepsilon$.

Similar to the observation in Example \ref{ex:Cantor}, and reminiscent of \cite{Sidorov}, we have 
\begin{thm}
Let $A$ be $n$-good.  
Then almost all $z \in [0,2]$ have a continuum of representations
    $z = x + y$ with $x, y \in C_A$.
\end{thm}

\begin{proof}
To see this note that $n-1 = 0 + (n-1) = (n-1) + 0$ has (at least) two
    representations as $n-1 = a_1 + a_2$ with  $a_1, a_2 \in A$.
We see that almost all $z \in [0,2]$ have infintely
    many $(n-1)$s in their base $n$ expansion with digits in $A+ A$.
As such, almost all $z \in [0,2]$ have a continuum of representations
    $z = x + y$ with $x, y \in C_A$.
\end{proof}

In the results above, we were interested in $0,n-1 \in A \in \{0,1,\dots,n-1\}$
    where $A$ was $n$-good.
That is, where $C_A + C_A = [0,2]$.

More generally, we can ask what the possible structures of $C_A + C_A$ can have.
In Theorem~\ref{thm:structure} we show
    that if $0, n-1\in A \subset \{0,1,\dots, n-1\}$ then the structure
    of $C_A + C_A$ is one of three possible shapes.
Namely, either it is a Cantor set, a full interval (i.e. $n$-good),
    or a countable collections of intervals and gaps.

In Section \ref{sec:comments}, we consider how the answers to the above questions
    change if we allow general $A \in \mathbb{Z}$, or we only require
    $C_A + C_A$ to contain an interval.

In the final section, Section \ref{sec:conc}, we give some final concluding
    remarks and indicate possible directions for future research.

\section{Results for small $A$}
\label{sec:small A}

In this section we consider how small we can have $A$ if $A$ is $n$-good.

\begin{thm}
\label{thm:lower bound good}
If $\#A < \sqrt{n}$, then $A$ is not $n$-good.
\end{thm}

\begin{proof}
It is trivial that $\underline{\dim}_B (B_1+B_2)\le \overline{\dim}_B B_1 + \overline{\dim}_B B_2$.
In our setting $B_1=B_2 = C_A$ 
    is self-similar and satisfies the open set condition, whence
\[
\overline{\dim}_B B_1 = 
\overline{\dim}_B C_A=\dim_H C_A=\frac{\log \#A}{\log n}< \frac{1}{2}.
\]
Hence $\underline{\dim}_B (C_A+C_A)<1$, so $A$ is not $n$-good.
\end{proof}

\begin{thm}
\label{thm:upper bound good}
For all $n$ there exists an $A$ with $\#A = \mathcal{O}(\sqrt{n})$ where $A$ is $n$-good.
\end{thm}

\begin{proof}
Choose $k \approx \sqrt{n}$.

Set
\begin{align*}
A_1 & = \{0,1,2,\dots,k\} \\
A_2 & = \{n-1,n-2,n-3,\dots,n-1-k\} \\
A_3 &= \{0,k,2k,3k,\dots,tk\},
\end{align*}
where $n-1-k \leq tk \leq n-1$.

We claim that $A = A_1 \cup A_2 \cup A_3$ is $n$-good.
As $\#A_i = \mathcal{O}(\sqrt{n})$, this will prove the result.

Consider $0 \leq a \leq 2n-2$.
We claim that $a \in A+A$.
First, assume $a \leq n-1$.
Write $a = k a_1 + a_2$ for some
    $a_1 \in \{0,1,\dots,t\}$ and
    $a_2 \in \{0,1,\dots,k\}$.
We see that $k a_1 \in A_3 \subset A$ and $a_2 \in A_1 \subset A$.
Hence $a$ is in $A+A$.

If $a \geq n$ we use a similar construction using $A_2$ instead of $A_1$.

This implies that the maximal distance between consecutive terms in
     $A+A$ is $1$.

As $C_A + C_A = C_{A+A, n}$ this suffices to prove that $A$ is $n$-good.
\end{proof}

\begin{corollary}
\label{cor:lower bound small UA}
For all $n$ there exists an $A$ with $\#A = \mathcal{O}(\sqrt{n})$ where $A$ is $n$-good and $U_A = \{0,2\}$.
\end{corollary}

\begin{proof}
The construction in Theorem \ref{thm:upper bound good} is an example
    of an $A$ with this property.
\end{proof}

\begin{example}
Let
\begin{align*}
A = & \{0,1,2,\dots,8,9, 10 \} \cup \{0,10, 20, 30, \dots, 80, 90, 100\} \cup\\ 
    & \{90, 91, 92,\dots,98,99,100\}.
\end{align*}
It is easy to check that $A+A = \{0,1,\dots, 200\}$.
Hence $A$ is $101$-good.
\end{example}

\section{Small non-trivial $U_A$}
\label{sec:small UA}

In this section we show that either $U_A$ is trivial, or the dimension
    of $U_A$ is bounded below.

Before proving this, we need to introduce some notation and common
    techniques for graph directed iterated function systems.

Let $G$ be a transitive directed graph with the set of vertices $V = \{1,2,\dots, n\}$.
We allow the directed graph to have loops and multiple edges between vertices.
To each edge we associate a linear contraction $S:\mathbb{R} \to \mathbb{R}$.
By \cite{MW}, there exists unique non-empty compact sets
    $K_1, K_2, \dots, K_n$ associated to each vertex such that
    \[ K_i = \bigcup S_{j,i}(K_j),\]
    where the union is taken over all vertices $j \in V$ and
    all edges mapping from vertex $j$ to vertex $i$.
We say that $K_i$ is the {\em attractor associated to vertex $i$}.
The digraph and associated contractions is called a {\em graph directed iterated
    function system (GDIFS)}, and the
    $K_i$ is the {\em attractor associated to vertex $i$}.
See \cite{MW} for further details.

\begin{thm}
\label{thm:dichotomy}
Assume that $A$ is $n$-good.
Either $\dim_H(U_A) \geq \log(2)/\log(n)$ or $U_A = \{0,2\}$.
\end{thm}

\begin{proof}
Let $n$ be fixed and $A = \{0 =: a_0 < a_1 < \dots < a_k := n-1\}$
    be $n$-good.

Partition $[0,2]$ into intervals $I_\ell = [\ell/n, (\ell+1)/n]$ for
    $\ell = 0, 1, \dots, 2n-1$.
We see for each pair $(a_i, a_j) \in A \times A$ such that
    $S_{a_i + a_j} ([0,2]) = I_{a_i + a_j} \cup I_{a_i + a_j + 1}$.
In particular, $S_{a_i+a_j}([0,1]) = I_{a_i + a_j}$ and 
    $S_{a_i + a_j}([1,2]) = I_{a_i + a_j +1 }$.

We see that each interval is covered by the left half of some image,
    or the right half of some image, or possibly both.
We wish to identify those that are covered uniquely by the left half of
    some image and by no right half, and similarly those that are covered uniquely by the
    right half of some image and no left half.

More precisely, we say an interval $I_\ell$ is {\em of type $L$} (for left) if
    there exists a unique pair $(a_i, a_j) \in A \times A$ such that
    $I_\ell = S_{a_i + a_j}([0,1])$, and that for all pairs
    $(a_i', a_j')$ we have $I_\ell \cap S_{a_i' + a_j'}((1,2)) = \varnothing$.
Similarly, an interval $I_\ell$ is {\em of type $R$} (for right) if
    there exists a unique pair $(a_i, a_j) \in A \times A$ such that
    $I_\ell = S_{a_i + a_j}([1,2])$, and that for all pairs
    $(a_i', a_j')$ we have $I_\ell \cap S_{a_i' + a_j'}((0,1)) = \varnothing$.
All other intervals will be {\em of type $O$} (for other).

If $I_\ell$ is of type $O$ then all points in $I_\ell$ have
    multiple representations, and hence are not in $U_A$.

Consider the graph directed iterated function system given by
\begin{align*}
\mathcal{L} & = \left(
                \bigcup_{\substack{S_\ell \text{ is of type $L$} \\
                          0 \leq \ell \leq n-1}}
                S_{\ell} (\mathcal{L}) \right)
                \cup
                \left(
                \bigcup_{\substack{S_\ell \text{ is of type $R$} \\
                          0 \leq \ell \leq n-1}}
                S_{\ell} (\mathcal{R}) \right) \\
\mathcal{R} & = \left(
                \bigcup_{\substack{S_\ell \text{ is of type $L$} \\
                          n \leq \ell \leq 2n-1}}
                S_{\ell} (\mathcal{L}) \right)
               \cup
                \left(
                \bigcup_{\substack{S_\ell \text{ is of type $R$} \\
                          n \leq \ell \leq 2n-1}}
                S_{\ell} (\mathcal{R}) \right) \\
\end{align*}
From \cite{MW} we see that $\dim_{H}(\mathcal{L}) = \dim_H(\mathcal{R})$.
We see that $U_A \cap [0,1] \subset \mathcal{L}$ and
    $U_A \cap [1,2] \subset \mathcal{R}$.
Hence $U_A \subset \mathcal{L} \cup \mathcal{R}$.
We have that $\left(\mathcal{L} \cup \mathcal{R}\right) \setminus U_A$
    is at most a countable number of points.
To see this, we note that the only points in $\mathcal{L} \cup \mathcal{R}$
    that are not in $U_A$ are those points with are images of $0$ or $2$
    under finite compositions of these maps.
Hence $\dim_{H}(U_A) = \dim_{H}(\mathcal{L}) = \dim_H(\mathcal{R})$.

Let \[ M = \begin{bmatrix}a & b \\ c & d \end{bmatrix}\]
    be the adjacency matrix for this graph directed iterated
    function system.
Here $a$ is the number of $\ell$ with $0 \leq \ell \leq n-1$ such that $I_\ell$ is of type
    $L$.
Similarly,
    $b$ is the number of $\ell$ with $0 \leq \ell \leq n-1$ such that $I_\ell$ is of type $R$,
    $c$ is the number of $\ell$ with $n \leq \ell \leq 2n-1$ such that $I_\ell$ is of type $L$, and
    $d$ is the number of $\ell$ with $n \leq \ell \leq 2n-1$ such that $I_\ell$ is of type $R$.

As $I_0$ is of type $L$ and $I_{2n-1}$ is of type $R$ we see that
    $a, d \geq 1$.
Hence the maximal eigenvalue of $M$ is greater than or equal to $1$.

If the Perron-Frobenius eigenvalue of $M$ is 1, then $a = d = 1$ and $b c = 0$.
Assume without loss of generality that $c = 0$.
If $b = 0$ then $U_A = \{0,2\}$ and we are done.
Hence assume that $b \geq 1$.

In this case $\mathcal{R} = \{2\}$.
For all $1 \leq \ell \leq n-1$ where $I_{\ell}$ is of type $R$ we
    see that the point $S_{\ell-1}(2) \in \mathcal{L}$.
Although these points are in $\mathcal{L}$, they are not points with
    unique representations.
To see this we note that $S_{\ell-1}(2)$ has address
    $(\ell-1) (2n-2) (2n-2) (2n-2) \dots$.
As there are no intervals $I_\ell$ with $1 \leq \ell \neq 2n-2$
    of type $L$ we see that $\ell \in A+A$.
Hence this point also has address $(\ell) 0 0 0 \dots$.
As this point has at least two representation, it is not in $U_A$.
Hence $\mathcal{L} = \{0\}$, and  so $U_A = \{0,2\}$.

Recall that $a, d \geq 1$ and $b, c$ are non-negative integers.
If the Perron-Frobenius eigenvalue $\lambda$ of $M$ is greater than $1$,
    then $b, c \geq 1$ and hence $\lambda \geq 2$.
This gives us that
     $\dim_H(U_A) = \dim_H(\mathcal{L}) =\dim_H(\mathcal{R}) = \log(\lambda)/\log(n) \geq \log(2)/\log(n) > 0$, and the result follows.
\end{proof}

\begin{example}
\label{ex:UK8}
Consider $A = \{0,2,5,7\}$.
We see that $A + A = \{0,2,4,5,7,9,10,12,14\}$.
It is worth noting that $0, 4, 10, 14$ all have unique representations
    as $a + a'$ with $a, a' \in A$.
As the maximal distance between consecutive terms is $2$, we see that
    $A$ is $8$-good.
Subdividing $[0,2]$ into $16$ intervals, we see that
\begin{align*}
I_0, I_1, \dots, I_7 & = L, R, O, O, L, O, O, O \\
I_8, I_0, \dots, I_{15} & = O, O, O, R, O, O, L, R
\end{align*}
From this we see that
\begin{align*}
\mathcal{L} & = S_{0} (\mathcal{L}) \cup S_{0} (\mathcal{R}) \cup
                S_{4} (\mathcal{L}) \\
\mathcal{R} & = S_{10} (\mathcal{R}) \cup S_{14} (\mathcal{L}) \cup
                S_{14} (\mathcal{R}).
\end{align*}
We can represent this by the directed graph in Figure \ref{fig:8verygood}.

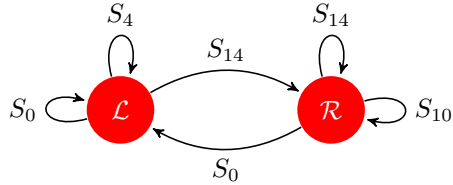
\begin{figure}[H]
\begin{center}
\begin{tikzpicture}[->,>=stealth',shorten >=1pt,auto,node distance=2.8cm, semithick]
  \tikzstyle{every state}=[fill=red,draw=none,text=white]

  \node[state]         (A)              {$\mathcal{L}$};
  \node[state]         (B) [right of=A] {$\mathcal{R}$};

  \path (A) edge [loop left]  node {$S_0$}         (A)
            edge [loop above]  node {$S_4$} (A)
            edge [bend left]  node {$S_{14}$} (B)
        (B) edge [loop right]  node {$S_{10}$}         (B)
            edge [loop above]  node {$S_{14}$} (B)
            edge [bend left]  node {$S_{0}$} (A);
\end{tikzpicture}
\end{center}
\caption{GDIFS Diagram}
\label{fig:8verygood}
\end{figure}

The incidence matrix of the graph directed IFS is $\begin{bmatrix} 2 & 1 \\ 1 & 2 \end{bmatrix}$, which has a maximal eigenvalue of $\lambda = 3$.
This gives us that $\dim_H(U_A) = \frac{\log(\lambda)}{\log(8)} =
    \frac{\log(3)}{\log(8)} \approx 0.52832$.
\end{example}

\begin{corollary}
\label{cor:upper bound large UA}
Let $A$ be $n$-good.
If $1, n-2 \not\in A$ then $\dim_H U_A \geq \log(2)/\log(n) > 0$.
\end{corollary}

\begin{proof}
Assume $1, n-2 \not\in A$.
Using the notation above, we see that $I_0$ is type $L$, $I_1$ is type $R$,
    $I_{2n-2}$ is type $L$ and $I_{2n-1}$ is type $R$.
Hence incidence matrix is strictly positive with integer values.
The maximal eigenvalue of the adjacency matrix associated to the graph
    directed iterated function system is hence bounded below by $2$.
This proves the result.
\end{proof}

\begin{example}
In Example \ref{ex:UK8} we see that $I_0, I_1 = I_{14},I_{15} = L, R$.
This gives a lower bound on the dimension of $\log(2)/\log(8)=\frac13$, although
    it is actually higher in this case.
\end{example}

\section{Large $U_A$}
\label{sec:large UA}
In this section we focus on the question: How large can $\dim_H(U_A)$ be?
Computationally this appears to be bounded above by $\log(2)/\log(3)$.
We show that we can get arbitrarily close to $\log(2)/\log(3)$ (excluding trivial
$n=3^k$) as
    $n$ tends to infinity.
First we need a definition.

\begin{definition}
Let $A$ be an $n$-good set and $M = \begin{bmatrix}a & b \\
    c & d \end{bmatrix}$
    be the adjacency matrix for this graph directed iterated
    function system representing $U_A$.
We will say that $A$ is $n$-very-good if
\begin{itemize}
\item $A$ is $n$-good
\item $1 \not\in A$ and $n-2\not\in A$.
\item Either $a+b =c+d$ or $a+c=b+d$.
\end{itemize}
\end{definition}

\begin{example}
We have that $A = \{0,2,5,7\}$ from Example \ref{ex:UK8} is $8$-very-good.
\end{example}

\begin{thm}
\label{thm:UK}
Let $A = \{a_0, a_1, \dots, a_k\}$ be $n$-very good.
Let
\begin{align*}
A^0 &= \{a_0, a_1, \dots, a_k, a_0+2n, a_1+2n, \dots, a_k+2n\},\\
A^1 &= \{a_0, a_1, \dots, a_k, a_0+2n-1, a_1+2n-1, \dots, a_k+2n-1\},\\
A^2 &= \{a_0, a_1, \dots, a_k, a_0+2n-2, a_1+2n-2, \dots, a_k+2n-2\}.
\end{align*}
Then $A^0$ is $(3n)$-very-good, $A^1$ is $(3n-1)$-very-good and $A^2$ is
    $(3n-2)$-very-good.

Further, if $\dim_H(U_A) = \frac{\log(\lambda)}{\log(n)}$ then
\begin{align*}
\dim_H(U_{A^0}) & = \frac{\log(2\lambda)}{\log(3n)} \\
\dim_H(U_{A^1}) & = \frac{\log(2\lambda-1)}{\log(3n-1)} \\
\dim_H(U_{A^2}) & = \frac{\log(2\lambda-2)}{\log(3n-2)}
\end{align*}
\end{thm}

\begin{example}
Consider $A = \{0,2,4\}$.

One can check that $A$ is $5$-very-good with
\[ I_0,\dots,I_9 = L, R, O, O, O, O, O, O, L, R, \]
adjacency matrix $\begin{bmatrix}1 & 1 \\ 1 & 1 \end{bmatrix}$ 
and $\dim_H U_A=\log(2)/\log(5)$.

We have that
\begin{align*}
A^0 & = \{0,2,4,10,12,14\} \\
A^1 & = \{0,2,4,9,11,13\} \\
A^2 & = \{0,2,5,8,10, 12\}
\end{align*}

One can check that $A^0$ is $15$-very-good with
\[ I_0,\dots,I_{14}    =  L, R, O, O, O, O, O, O, L, R, O, O, O, O, O \]
\[ I_{15},\dots,I_{29} =  O, O, O, O, O, L, R, O, O, O, O, O, O, L, R \]
with adjacency matrix $\begin{bmatrix}2 & 2 \\ 2 & 2 \end{bmatrix}$
and $\dim_H U_A=\log(2\cdot 2)/\log(3\cdot 5) = \log(4)/\log(15)$.

Similarly, $A^1$ is $14$-very-good with
\[ I_0,\dots,I_{13}    =  L, R, O, O, O, O, O, O, L, O, O, O, O, O \]
\[ I_{14},\dots,I_{27} =  O, O, O, O, O, R, O, O, O, O, O, O, L, R \]
with adjacency matrix $\begin{bmatrix}2 & 1 \\ 1 & 2 \end{bmatrix}$
and $\dim_H U_A=\log(2\cdot 2-1)/\log(3\cdot 5-1) = \log(3)/\log(14)$.

Finally, $A^2$ is $13$-very-good with
\[ I_0,\dots,I_{12}    =  L, R, O, O, O, O, O, O, O, O, O, O, O \]
\[ I_{13},\dots,I_{25} =  O, O, O, O, O, O, O, O, O, O, O, L, R \]
with adjacency matrix $\begin{bmatrix}1 & 1 \\ 1 & 1 \end{bmatrix}$
and $\dim_H U_A=\log(2\cdot 2-2)/\log(3\cdot 5-2) = \log(2)/\log(13)$.
\end{example}

\begin{proof}[Proof of Theorem \ref{thm:UK}]
We will do the case of $A^1$ only.
The other cases are similar.
Assume that $A = \{a_0, a_1, \dots, a_k\}$ is $n$-very-good.
We know that $a_1 \neq 1$ and $a_{k-1} \neq n-2$ by assumption.
Let the graph directed iterated function system used to determine
    $U_A$ have incidence matrix $\begin{bmatrix} a& b \\ c & d \end{bmatrix}$.
Consider $A^1 = \{a_0, a_1, \dots, a_k, a_0+2n-1, a_1+2n-1, \dots, a_k+2n-1\}$.
We have that
\[ A^1 + A^1 = (A+ A) \cup (A + A + (2n-1)) \cup (A + A + (4n-2)). \]
The maximal term in $A + A$ is $2n-2$ by construction.
The minimal term in $A + A + (2n-1)$ is $2n-1$ by construction.
This gives us that $I_{2n-1}$ is not type $R$.
We further see that there are $a + b$ intervals $I_\ell$ with 
    $0 \leq \ell \leq 2n-1$ such that $I_\ell$ is of type $L$.
There are $c + d -1$ intervals $I_\ell$ with $0 \leq \ell \leq 2n-1$
    such that $I_\ell$ is of type $R$.

We see that every term in $A + A + (2n-1)$ has at least two representations.
Hence all $I_{2n}, I_{2n+1}, \dots, I_{4n-2}$ are of type $O$.

Similar to before, we have that $I_{4n-1}$ is type $O$.
As before, there are $a + b - 1$ intervals $I_\ell$ with 
    $4n-1 \leq \ell \leq 6n-3$ such that $I_\ell$ is of type $L$.
There are $c + d$ intervals $I_\ell$ with $4n-1 \leq \ell \leq 6n-3$
    such that $I_\ell$ is of type $R$.

This gives us that the incidence matrix for $U_{A^1}$ is
    $\begin{bmatrix} a+b & c + d -1 \\ a + b - 1 & c + d \end{bmatrix}$.
As $1, 3n-4 \not\in A^1$ we see that $A^1$ is $(3n-1)$-very-good.

We next need to compute the dimension of $U_{A^1}$.

Consider the incidence matrix for $U_A$.
We have that either $a +b  = c + d$ or $a + c = b + d$ as $A$ is $n$-very-good.
Assume that $a+b = c +d$.  The other case is similar.
We see that the maximal eigenvalue of the incidence matrix is $a+b$.
This gives us that $\lambda = a + b = c + d$.
We see that the two eigenvalues of 
$\begin{bmatrix} a + b& c +d -1 \\ a+ b -1 & c +d  \end{bmatrix} = 
\begin{bmatrix} \lambda& \lambda -1 \\ \lambda -1 & \lambda  \end{bmatrix}$ 
are $a + b + c +d -1 = 2 \lambda -1$ and $1$.
Hence $\dim_{U_{A^1}} = \log(2 \lambda -1 )/\log(3n-1)$ as required.
\end{proof}

\begin{corollary}
\label{cor:UA upper}
There exists a sequence of $A_n$ which are $n$-very good and
    $\lim \dim_H(U_{A_n}) = \log(2)/\log(3)$.
\end{corollary}

\begin{proof}
Let $N = 3$.
We first note for $n_0 \in [9,27] = [3^{N-1},3^N]$ that there exists an
    $n$-very-good set with $\dim_H(U_{A_n}) \geq d$ where
    $d = 0.442144$.
See Table \ref{tab:UK}.

\begin{table}
\begin{tabular}{lll}
$n$ & $A$ & $\dim_H(U_A)$ \\
\hline
9 & [0, 2, 6, 8] & .6309297534\\
10 & [0, 2, 6, 7, 9] & .4771212549\\
11 & [0, 2, 4, 8, 10] & .4581569101\\
12 & [0, 2, 3, 5, 9, 11] & .4421141088\\
13 & [0, 2, 6, 10, 12] & .5404763090\\
14 & [0, 2, 6, 7, 11, 13] & .5252990700\\
15 & [0, 2, 6, 8, 12, 14] & .5119160496\\
16 & [0, 2, 6, 9, 13, 15] & .5000000000\\
17 & [0, 2, 6, 10, 14, 16] & .4893010842\\
18 & [0, 2, 6, 7, 11, 15, 17] & .4796249332\\
19 & [0, 2, 4, 10, 12, 16, 18] & .5466025696\\
20 & [0, 2, 3, 5, 12, 14, 17, 19] & .4627564262\\
21 & [0, 2, 3, 5, 12, 14, 18, 20] & .5286339466\\
22 & [0, 2, 5, 7, 13, 15, 19, 21] & .5206780355\\
23 & [0, 2, 6, 8, 14, 16, 20, 22] & .5714440358\\
24 & [0, 2, 6, 8, 15, 17, 21, 23] & .5637914160\\
25 & [0, 2, 6, 8, 16, 18, 22, 24] & .5566413765\\
26 & [0, 2, 6, 8, 17, 19, 23, 25] & .5972536806\\
27 & [0, 2, 6, 8, 18, 20, 24, 26] & .6309297534
\end{tabular}
\caption{Table of $n$-very-good sets $A$ with $\dim_H(U_A)$}
\label{tab:UK}
\end{table}

By Theorem \ref{thm:UK},
    for all $n_1 \in [3^N, 3^{N+1}]$ there exists
    an $n_0 \in [3^{N-1}, 3^N]$ and a $k_1 \in \{0,1,2\}$ such that
    $A_{n_1} = (A_{n_0})^{k_1}$ is $n_1$-very-good.
In general, for all $n_t \in [3^{N+t-1}, 3^{N+t}]$ there exists
    an $n_0 \in [3^{N-1}, 3^N]$ and
    a sequence $k_1, k_2, \dots, k_t \in \{0,1,2\}$ such that
    \[ A_{n_t} = \left(\dots\left(\left(A_{n_0}\right)^{k_1}\right)^{k_2}\dots\right)^{k_t} \]
    is $n_t$-very-good.

Let $d = \dim_H(U_{A_{n_0}})$.
Hence $n^d = \lambda$.
For $k_1 \in \{0,1,2\}$ we have that
\begin{align*}
\dim_H(U_{A_{n_0}^{k_1}})
    & = \frac{\log(2 n^d - k_1)}{\log(3 n-k_1)} \\
    & = \frac{\log(2) + d \log(n) + \log(1 - \frac{k_1}{2n^d})}
             {\log(3) + \log(n) + \log(1 - \frac{k_1}{3n})} \\
\end{align*}
And further, by induction,
\begin{align*}
\dim_H(U_{(\dots((A_{n_0}^{k_1})^{k_2})\dots)^{k_t}})
    & = \frac{\log(2(\dots (2 (2 n^d - k_1)-k_2)\dots)-k_t)}
             {\log(3(\dots (3 (3 n - k_1)-k_2)\dots)-k_t)} \\
    & = \frac{\log(2^t n^d - 2^{t-1} k_1 - 2^{t_2} k_2- \dots- k_t)}
             {\log(3^t n - 3^{t-1} k_1 - 3^{t_2} k_2-\dots - k_t)} \\
    & = \frac{\log(2^t n^d (1- \frac{k_1}{2n^d} - \frac{k_2}{2^2n^d}- \dots- \frac{k_t}{2^t n^d})}
            {\log(3^t n (1- \frac{k_1}{3n} - \frac{k_2}{3^2n}- \dots- \frac{k_t}{3^t n})}
\end{align*}

Denote
    \[ 1-x = 1- \frac{k_1}{2n^d} - \frac{k_2}{2^2n^d}- \dots- \frac{k_t}{2^t n^d} \]
and
    \[ 1-y = 1- \frac{k_1}{3n} - \frac{k_2}{3^2n}- \dots- \frac{k_t}{3^t n}.\]
We see that
\[
|x| \leq \sum_{i=1}^t \frac{k_i}{n^d\cdot 2^i}
  \leq \sum_{i=1}^\infty \frac{2}{9^{0.442144}\cdot 2^i}
   \leq 0.75708
\]
and
\[
|y| \leq \sum_{i=1}^t \frac{k_i}{n\cdot 3^i}
  \leq \sum_{i=1}^\infty \frac{2}{9\cdot 3^i}
   \leq \frac{1}{9}
\]
respectively.

Hence $\log(1-x)$ and $\log(1-y)$ are well defined and bounded.

This gives us that
\[
\dim_H(U_{(\dots((A_{n_0}^{k_1})^{k_2})\dots)^{k_t}})
    = \frac{t \log 2 + d \log n + \log (1-x) }
            {t \log 3 + \log n + \log (1-y) }.
\]
As $t \to \infty$ we have the dimension goes to $\log(2) / \log(3)$, as
    required.
\end{proof}

\begin{example}
Consider $n = 1000000$.
We note that $A_{1000000} = (A_{333334})^2$, hence if we can find a
    very-good $A$ for $A_{333334}$ we can find a very-good $A$
    for $A_{1000000}$.
This technique can be applied recursively.
See Table~\ref{tab:1000000} for full details.
\begin{table}
\begin{tabular}{ll}
$A_i$ & $\dim_H(U_A)$ \\
\hline
$A_{17}$ & $\frac{\log(4) }{\log(17)} \approx .4894$\\
$A_{51} = (A_{17})^0$ & $\frac{\log(8) }{\log(51)} \approx .5289$\\
$A_{153} = (A_{51})^0$ & $\frac{\log(16) }{\log(153)} \approx .5512$\\
$A_{458} = (A_{153})^1$ & $\frac{\log(31) }{\log(458)} \approx .5605$\\
$A_{1372} = (A_{458})^2$ & $\frac{\log(60) }{\log(1372)} \approx .5667$\\
$A_{4116} = (A_{1372})^0$ & $\frac{\log(120) }{\log(4116)} \approx .5752$\\
$A_{12346} = (A_{4116})^2$ & $\frac{\log(238) }{\log(12346)} \approx .5808$\\
$A_{37038} = (A_{12346})^0$ & $\frac{\log(476) }{\log(37038)} \approx .5860$\\
$A_{111112} = (A_{37038})^2$ & $\frac{\log(950) }{\log(111112)} \approx .5900$\\
$A_{333334} = (A_{111112})^2$ & $\frac{\log(1898) }{\log(333334)} \approx .5935$\\
$A_{1000000} = (A_{333334})^2$ & $\frac{\log(3794) }{\log(1000000)} \approx .5965$
\end{tabular}
\caption{Construction of good $A_{1000000}$}
\label{tab:1000000}
\end{table}
\end{example}

\begin{remark}
This shows that $\limsup_n \max_A \dim_H(U_A) \geq \log(2)/\log(3)$.
This does not show equality, as we only know that $\dim_H(U_A)$ is
    bounded above by $1$.
\end{remark}

\begin{remark}
\label{rem:UA upper}
Extensive computations have been done to attempt to find an $A$ with
    $0, n-1 \in A$, $A$ $n$-good, and $\dim_H(U_A) > \log(2)/\log(3)$.
This search has been unsuccessful.
For each $n$ in Figure~\ref{fig:UA upper} we have given the largest dimension
    known for $\dim_H(U_A)$.
For reference, we have put a horizontal line at $\log(2)/\log(3)$.
It is worth noting that this search is not exhaustive (as the number of sets
    are too large).
A complete data set for $3 \leq n \leq 1000$ can be found at \cite{Homepage}.

It is also worth noting that if this search is repeated for
    {\em all} $0, n-1 \in A \subset \{0,1,\dots, n-1\}$,
    including $A$ where $\dim_H(C_A + C_A) < 1$,
    we still cannot find an $A$ such that $\dim_H(U_A) > \log(2)/\log(3)$.
\end{remark}

\begin{figure}
\includegraphics[scale=0.7, angle=270]{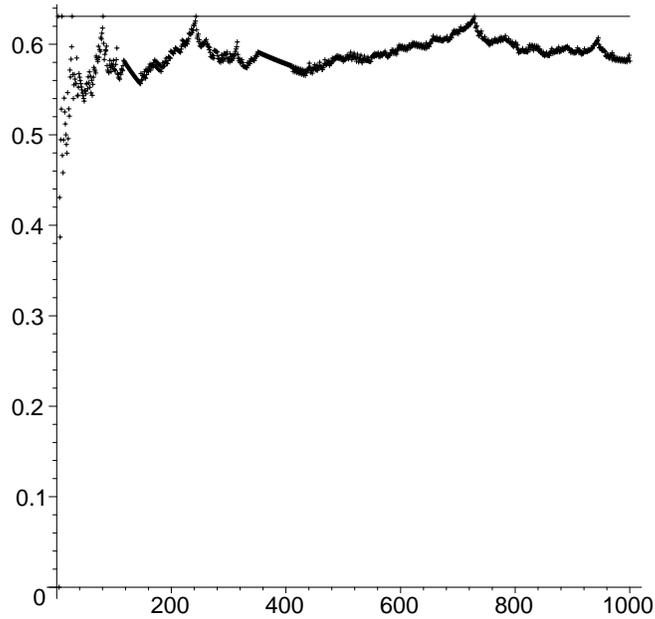}
\caption{Largest known $\dim_H(U_A)$ for $0, n-1 \in A$, with $A$ $n$-good.}
\label{fig:UA upper}
\end{figure}

\section{Possible structures for $C_A + C_A$}
\label{sec:structure}

If $A$ is $n$-good, by definition, $C_A + C_A = [0,2]$.
In this section we consider the structure of $C_A + C_A$ when $A$ is not
    $n$-good.

When $0, n-1 \in A \subset \{0,1,\dots,n-1\}$ we have examples where
   $C_A + C_A = [0,2]$.
Further, for $A = \{0, n-1\}$ and $n \geq 4$ we have that $C_A + C_A$
    is a Cantor set of dimension $\log(3)/\log(n)$.
Below is an example of an $A$ such that $C_A + C_A$ is not $n$-good,
    but where $C_A + C_A$ contains an interval, and hence has dimension $1$.

\begin{example}
\label{ex:v2}
Let $A = \{0,1,4\}$.
We observe that $C_A + C_A \neq [0,2]$ for the simple reason that $(7/5, 8/5) \cap (C_A + C_A) = \varnothing$.
This comes directly from noting that $C_A \subset [0,2]$ and
    $(7/5, 8/5) \cap \left(\cup_{a, a' \in A} S_{a+a'}([0,2])\right) = \varnothing$.
Hence $A$ is not $5$-good.
We also observe that $[1, 5/4] \subset C_A + C_A$.
This can be shown using techniques in \cite{NgaiWang01} to determine
    the structure of $C_{A+A, n}$.
Hence $C_A + C_A$ has dimension $1$.

We will say that $I = [a,b]$ is an {\em interval} in $C_A + C_A$ is $I \subset C_A + C_A$.
We will say that $I$ is a {\em maximal interval} if $I$ is an interval of $C_A + C_A$ and $I$ is not
    the proper subset of any other interval in $C_A + C_A$.
For example, $[1, 5/4]$ is a maximal interval.

We will similarly say that $G = (a,b)$ is a {\em gap} in $C_A +C_A$ if $G \cap (C_A + C_A) = \varnothing$.
We will say that $G$ is a {\em maximal gap} if $G$ is an gap of $C_A + C_A$ and $G$ is not
    the proper subset of any other gap in $C_A + C_A$.
For example, $(7/5, 8/5)$ is a maximal gap.

We will say that $g \in C_A + C_A$ is a point in $C_A + C_A$ if both $(g-\epsilon, g)$ and $(g, g+ \epsilon)$ have
    non-trivial intersection with $C_A + C_A$ and non-trivial intersection with the compliment of
    $C_A + C_A$.

In this case it can be shown that $C_A + C_A$ is composed of a countably infinite number of maximal gaps,
    a countably infinite number of maximal intervals and an uncountable number of points.
\end{example}

This is in fact a general phenomenon.
\begin{thm} \label{thm:structure}
Let $0, n-1 \in A \subset \{0,1,\dots, n-1\}$.
Define $C_A$ as the IFS generated by
\[ \{ x \mapsto (x + a) / n: a \in A \}. \]
Then one of the following is true.
\begin{enumerate}
\item $C_A + C_A = [0, 2]$. \label{case:1}
\item $C_A + C_A$ is a Cantor set.  (I.e. every point is a boundary point and no isolated points.) \label{case:2}
\item $C_A + C_A$ contains a countably infinite number of maximal intervals
      and a countably infinite number of maximal gaps.
      Furthermore, the set of points in $C_A + C_A$ has positive dimension.
      \label{case:3}
\end{enumerate}
\end{thm}

\begin{proof}
We have seen examples of all three of these possibilities.
Hence, it suffices to show that if neither Cases \eqref{case:1} or \eqref{case:2}
    hold, then Case \eqref{case:3} holds.

Let $A = \{a_0 < a_1 < … < a_k\}$ where $a_0 = 0$ and $a_k = n-1$.
Let $B = A + A = \{b_0 < b_1 < .. < b_j\}$ where $b_0 = 0$ and $b_j = 2n-2$.

We see that $C_A + C_A$ is the attractor of the IFS \[ \{S_i(x) =  x/n + b_i/n\}_{i=0}^{j}.\]

Assume that $C_A + C_A$ is not an interval and is not a Cantor set.
Then $C_A + C_A$ will contain an interval (say $[a,b]$) and will contain a gap (say $(c,d)$).

We can assume without loss of generality that either $1 < c < d$ or $c < d < 1$ by shrinking the gap if necessary.
Assume without further loss of generality that $1 < c < d$, as the other argument is symmetric.

We see that
\begin{itemize}
\item $S_{2n-2} \circ S_{2n-2} \circ \dots \circ S_{2n-2} (c,d) = S_{2n-2}^{[m]}((c,d))$ is a gap in $C_A + C_A$.
\item $S_{2n-2} \circ S_{2n-2} \circ \dots \circ S_{2n-2}([a,b]) = S_{2n-2}^{[m]}([a,b])$ is an interval in $C_A + C_A$.
\end{itemize}

This shows that we have a countably infinite sequence of intervals and a countably infinite sequence of gaps both approaching 2.
These two sequences interweave.
Hence we have at least a countably infinite  number of maximal intervals and a countably infinite number of maximal gaps.

For any $k \in \mathbb{N}$, $k \geq 2$ we see that we can have at most
    $2k$ maximal intervals of length at least $1/k$.
As such, we can enumerate the maximal intervals, and hence the number of
    maximal intervals is at most countable.
A similar result holds for maximal gaps.
This proves that we have a countably infinite number of maximal intervals and a countably infinite number of 
    maximal gaps.

Note that 2 is not contained in an interval, nor it is the boundary of a gap (from the left).
Let $\cup_{k} J_k$ be the disjoint union of all maximal intervals in
    $C_A + C_A$.
From above, we have that $2 \in (C_A + C_A) \setminus \cup_k J_k$.
Let $P = (C_A + C_A) \setminus \cup_k J_k$.

We will next show that $\dim_H(P) \geq \log(2)/\log(n)$.

Consider $C_{A+A,n}$.
As in the proof of Theorem \ref{thm:dichotomy} we will subdivide $[0,2]$
    into $2n$ intervals $I_\ell = [\ell/n, (\ell+1)/n]$ of size $1/n$.
In Theorem \ref{thm:dichotomy} we say an interval was of type $L$ if
    there existed a unique pair $(a_i, a_j) \in A \times A$ such that
    $I_\ell = S_{a_i + a_j}([0,1])$, and that for all pairs
    $(a_i', a_j')$ we have $I_\ell \cap S_{a_i' + a_j'}((1,2)) =
    \varnothing$.

Here we are concerned with $C_{A+A,n}$ instead of $C_{A,n} + C_{A,n}$, so
    we modify this slightly.
Here we say that an interval is of type $L$ if there exist an $a \in A+A$
    such that $I_\ell =t S_{a,n}([0,1])$ and that for all $a' \in A+A$
    we have $I_\ell \cap S_{a',n}((1,2)) = \varnothing$.
We define an interval to be of type $R$ in the analogous way.
We denote all other intervals to be of type $O$.
If an interval is of type $O$ then either there exists $a, a' \in A + A$ 
    with $I_\ell = S_a([0,1]) = S_{a'}([1,2])$ or for all $a, a' \in A +A$
    we have $I_\ell \cap S_a((0,1)) = I_\ell S_{a'}((1,2)) = \varnothing$.

Proceeding as before, we see that $I_0$ is of type $L$, and $I_{2n-1}$ is of
    type $R$.
We see that as $C_{A+A,n}$ contains a gap, then there exists an $I_\ell$
    which is covered by no $S_{a,n}([0,2])$.
If we consider the interval $I_{\ell-1}$ it will be one of two types.
It will either be of type $O$ as it is covered by no $S_{a,n}([0,2])$,
    or it will be of type $R$.
If it is of type $O$, then we can repeat this observation on $I_{\ell-2}$.
Repeating this observation as necessary, we see that there exists an
    $\ell' < \ell$ such that $I_{\ell'}$ is of type $R$.
Similarly there exists an $\ell'' > \ell$ such that $I_{\ell''}$ is of type
    $L$.
This gives use that $I_0$ and $I_{\ell''}$ are of type $L$ and
    $I_{\ell'}$ and $I_{2n-1}$ are of type $R$.

As before, we can construct a graph directed iterated function system
    using these four maps.
We have three possible cases.
Either $\ell' < \ell'' \leq n-1$, or
       $\ell' \leq n-1 < n \leq  \ell''$ or
       $n \leq \ell' <  \ell''$.
We will give the first one only.  The rest are analogous.
In the case $\ell' < \ell'' \leq n-1$ we have
\begin{align*}
\mathcal{L} & = S_0(\mathcal{L}) \cup
              S_{\ell'}(\mathcal{R}) \cup
              S_{\ell''}(\mathcal{L}) \\
\mathcal{R} & = S_{2n-2}(\mathcal{R})
\end{align*}
See Figure \ref{fig:GDIFS} for a graphical representation.

\begin{figure}[H]
\begin{center}
\begin{tikzpicture}[->,>=stealth',shorten >=1pt,auto,node distance=2.8cm, semithick]
  \tikzstyle{every state}=[fill=red,draw=none,text=white]

  \node[state]         (A)              {$\mathcal{L}$};
  \node[state]         (B) [right of=A] {$\mathcal{R}$};

  \path (A) edge [loop left]  node {$S_0$}         (A)
            edge [loop above]  node {$S_{\ell''}$} (A)
        (B) edge node {$S_{\ell'}$} (A)
            edge [loop right]  node {$S_{2n-2}$} (B);
\end{tikzpicture}
\end{center}
\caption{GDIFS Diagram}
\label{fig:GDIFS}
\end{figure}
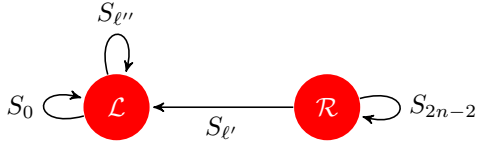

We see that $P \subset \mathcal{L} \cup \mathcal{R}$.
As before, there are at most a countable number of points in
    $\left(\mathcal{L} \cup \mathcal{R}\right) \setminus P$.
Hence $\dim_{H}(\mathcal{L}) = \dim_H(P)$.

Consider the adjacency matrix $M$ for the graph directed iterated function system.
We see that each column sum is at least 2, as there are at least two intervals
    of type $L$ and at least two intervals of type $R$.
Hence this adjacency matrix has an eigenvalue of at least $2$.
This implies that the attractor of the graph directed iterated function system
    has dimension at least $\log(2) / \log(n)$.
\end{proof}

It is worth noting that if $A$ is allowed to be an arbitrary set of integers then it is possible to have $C_A + C_A$ with a different structure.

\begin{example}
\label{ex:v4}
Let $n = 5$ and $A = \{0,1,7,8\}$.
We see that $B = A + A = \{0,1,2,7,8,9,14,15,16\}$.
In this case we can show that
    $C_{A,n} + C_{A,n} = [0,6/5] \cup [7/5, 13/5] \cup [14/5, 4]$.
\end{example}

It is unclear what the full range of possible structures of $C_{A,n} + C_{A,n}$
    when $A$ is not restricted to subsets of $\{0, 1, \dots, n-1\}$.

\section{Some comments on simplifying assumptions}
\label{sec:comments}

We made a number of simplifying assumptions in the initial definition of
   $n$-good.
The first was that $0, n-1 \in A \subset \{0,1,\dots,n-1\}$.
The second was that $C_A + C_A$ was an interval, instead of simply
    containing an interval.
In this section we consider how the results would be modified if these
    simplifying assumptions were relaxed.

\begin{definition}
{\ }
\begin{itemize}
\item
We say that $0,n-1 \in A \subset \{0,1,\dots,n-1\}$ is {\em $n$-good (v1)} if $C_{A,n} + C_{A,n} = [0,2]$.
\item
We say that $A \subset \mathbb{Z}$ is {\em $n$-good (v2)} if $C_{A,n} + C_{A,n}$ is an
    interval.
\item
We say that $0,n-1 \in A \subset \{0,1,\dots,n-1\}$ is {\em $n$-good (v3)} if $C_{A,n} + C_{A,n}$ contains an interval.
\item
We say that $A \subset \mathbb{Z}$ is {\em $n$-good (v4)} if $C_{A,n} + C_{A,n}$ contains an interval.
\end{itemize}
\end{definition}

Consider $A \subset \mathbb{Z}$.
It is worth observing that if we take a linear translate of $A$ then this results in a linear translate of $C_{A,n}$ and
     a linear translate of $C_{A,n} + C_{A,n}$.
As such, all answers to the structure questions remain the same under linear translates of $A$.
Hence, we will assume without loss of generality that $\min(A) = 0$.

We see from Example \ref{ex:v2} that the $A = \{0,1,4\}$ is $5$-good (v3) and (v4).
We see from Example~\ref{ex:v4} that the $A = \{0,1,7,8\}$ is $5$-good (v4).

We make a few comments upon the questions examined within this paper,
    with respect to these variations.
\begin{itemize}
\item How small can $A$ be if $A$ is $n$-good?
\begin{itemize}
\item For v2, v3, v4, we still have that $\mathcal{O}(\sqrt{n})$ is an
    attainable lower bound on the size of $A$.
    This is because the proof of Theorem~\ref{thm:lower bound good} uses
    the fact that $\dim_H(C_A + C_A) = 1$ (and hence $C_A + C_A$ contains an interval),
    and not that $C_A + C_A = [0,2]$.
    Theorem~\ref{thm:upper bound good} is an existence proof, and the example
    given is also an example in v2, v3 and v4.
\end{itemize}
\item How large can $A$ be if $A$ is not $n$-good?
\begin{itemize}
\item Unlike Remark \ref{rem:not n good}, for v2,
    we can have $A$ arbitrarily large. For any $m$ we can choose
    $k$ sufficiently large (with respect to $n$ and $m$)
    so that $A = \{0,1,2,\dots, m-1,  m, k\}$ is not $n$-good (v2).
\item For v3 consider $n = 2m$ or $n = 2m+1$.
    Then $A = \{0,1,2,\dots,m-2, n-1\}$ is not $n$-good and
    $\#A = \mathcal{O}(n)$.
\item For v4 let $n \geq 4$.
      Let $A = \{n^i\}_{i=0}^m$.
      Then $A$ is not $n$-good, regardless of the size of $m$.
      To see this, note that 
      \begin{align*}
       C_{A,n} & = \left\{\sum \frac{n^{a_i}}{n^i}: a_i \in \{0,1,\dots m\} \right\}  \\
             & = \left\{\sum \frac{1}{n^{i-a_i}}: a_i \in \{0,1,\dots m\} \right\}.
      \end{align*}
      Hence elements in $C_{A,n} + C_{A,n}$ can be written as
      \[ \sum_{i=1}^\infty \frac{1}{n^{i-a_i}} + \frac{1}{n^{i-b_i}}  \]
      with $a_i, b_i \in \{0,1,\dots,m\}$.
      In this case we cannot have more than $2 (m+1)$ consecutive $3$s
          in the base $n$ representation of this number.
      As numbers with $2(m+1)+1$ consecutive $3$ in their base $n$ expansion
          are dense in $\mathbb{R}$,
          this proves that the set contains no intervals.

      If instead $n = 3$ it is easy to show that $C_A + C_A$ is always an 
          interval, so long as $\#A \geq 2$.
\end{itemize}
\item How small can $A$ be with $A$ being $n$-good and $U_A = \{0,2\}$?
\begin{itemize}
\item The proof of Corollary~\ref{cor:lower bound small UA} is constructive,
      and depends only on $\dim_H(C_A + C_A) = 1$ (and hence $C_A + C_A$ contains an interval).
      Hence this example still holds for v2, v3 and v4.
      The lower bound given in Theorem~\ref{thm:lower bound good} is valid for
      all variations.
\end{itemize}
\item How large can $A$ be with $A$ being $n$-good and $\dim_H(U_A) > 0$?
\begin{itemize}
\item For v2 and v4, and $n \geq 6$ we have that $A$ can be arbitrarily large.
      To see this, let $t \geq 1$.
      Let
\begin{align*}
A  = & \{0\} \cup \{t(n-1)\} \cup \\
     & \{2t, 2t+1, \dots, t(n-1)-2t-1, t(n-1)-2t\}.
\end{align*}
      We have that
\begin{align*}
A + A = & \{0\} \cup \{2t (n-1)\} \cup  \\
        & \{2t, 2t+1, \dots, 2t(n-1)-2t-1, 2t(n-1)-2t\}.
\end{align*}
      This gives us that $C_A + C_A = [0, 2t]$, and hence $A$ is $n$ good for both v2 and v4.
      We further see that both $0$ and $t(2n-2)$ have unique representations in $A+A$.
      Further, the maps $x \mapsto x / n$ and $x \mapsto x/n + t(2n-2)/n$ acting on $(0,2t)$
          are disjoint from the action of all other maps.
      Hence any infinite non-trivial composition of these two maps results in a point in $U_A$.
      Hence $\dim_H(U_A) > 0$.
      As $A = \mathcal{O}(t (n-5))$ and $n \geq 6$ the result follows.
      It is unclear what happens in the case when $3 \leq n \leq 5$.
\item The result of Corollary \ref{cor:upper bound large UA} is still valid
      for v3, and hence the same example holds.
\end{itemize}
\item Does there exist $\{0,2\} \subsetneq U_A $ with
    $\dim_H(U_H) = 0$.
\begin{itemize}
\item We conjecture this is not true for v3 and v4 as stated.  There might be a
    modification that is true.
\item We conjecture that this is true for v2, (with the obvious modification).
\end{itemize}
\item How large can we make $\dim_H(U_A)$ if $A$ is $n$-good?
\begin{itemize}
\item We do not know the answer to this for any of the variations.
      We conjecture the upper bound is $\log(2)/\log(3)$ for all variations, with or without the requirement that $A$ is $n$-good.
\end{itemize}
\end{itemize}

It was remarked in Theorem \ref{thm:dichotomy} that, for v1, either $\dim_H(U_A) = 0$ or $\dim_H(U_A) \geq \log(2)/\log(n)$.
For v2 we can have $A$ as an arbitrary set of integers.
If $C_A + C_A = [0,2t]$ for some integer $t$, then the construction of the graph directed iterated function
    system in the proof in Theorem~\ref{thm:dichotomy} can be replaced by a graph directed iterated function
    system with $2t$ nodes.
Hence, this becomes a question the maximal eigenvalue of a non-negative $2t \times 2t$ matrix with integer
    coefficients.
It is clear that for any fixed $t$ the set of possible eigenvalues greater than $1$ are bounded
    away from $1$.
What is not clear is what exactly that bound is, and if it is achievable for all $t$.
We have that the result for v3 is the same as that for v1, and similarly v4 is the same as that for v2.

\section{Conclusions and Open questions}
\label{sec:conc}

In this paper we started the investigation of the Minkowski sum of two
    linear Cantor sets.
We said that a linear Cantor set $C$ was $n$-good if the sum $C+C$ was an
    interval (or contained an interval).
We considered how large or small $C$ could be and still have $C$ as a $n$-good
    or not $n$-good Cantor set.
We considered the set of points $U$ which had a unique representation in
    $C + C$.
Again, we considered how large or small $C$ could be and maintain certain
    properties about $U$.

An interesting, and still unresolved questions is: how big can $U$ be?
It appears computationally that $\dim_H(U) \leq \log(2)/\log(3)$,
    with this bound only being achieved at powers of $3$, and this
    upper bound being approached for large $n$.
We conjecture that $\dim_H(U) \leq \log(2)/\log(3)$ with equality
    only if $n= 3^k$ for some $k$.

A second interesting question, which was not considered, was higher
    sums.
For example, what can be said about $C+C+C$?


\providecommand{\bysame}{\leavevmode\hbox to3em{\hrulefill}\thinspace}
\providecommand{\MR}{\relax\ifhmode\unskip\space\fi MR }
\providecommand{\MRhref}[2]{%
  \href{http://www.ams.org/mathscinet-getitem?mr=#1}{#2}
}
\providecommand{\href}[2]{#2}

\end{document}